\documentclass[a4paper,10pt]{amsart}
\usepackage{amssymb,amsmath,amsfonts,amsthm}
\usepackage[dvips]{graphicx}
\usepackage{psfrag}

\theoremstyle{plain}
\newtheorem{main}{Theorem}

\newtheorem{theorem}{Theorem}[section]
\newtheorem{lemma}[theorem]{Lemma}
\newtheorem{proposition}[theorem]{Proposition}
\newtheorem{corollary}[theorem]{Corollary}
\theoremstyle{remark}
\newtheorem{remark}[theorem]{Remark}
\newtheorem{definition}[theorem]{Definition}
\newtheorem{example}[theorem]{Example}

\newcommand{\Leb}{\operatorname{vol}}

\newcommand{\C}{\operatorname{C}}
\newcommand{\ID}{\operatorname{ID}}
\newcommand{\Sing}{\operatorname{Sing}}

\newcommand{\bT}{\mathbb{T}}
\newcommand{\Per}{\operatorname{Per}}

           \def\ea{\end{array}}
          \def\ec{\end{center}}
     \def\ed{\end{description}}
        \def\ee{\end{equation}}
       \def\eea{\end{eqnarray}}
     \def\eeaa{\end{eqnarray*}}
 \def\et{\end{thebibliography}}
\def\bib{\bibitem}

\def\Orb{{\rm Orb}}

\def\Sing{{\rm Sing}}

\def\supp{\operatorname{supp}}

\def\cD{{\mathcal D}}

\def\cU{{\mathcal U}}
\def\cV{{\mathcal V}}

\def\cM{{\mathcal M}}
\def\cN{{\mathcal N}}
\def\cP{{\mathcal P}}
\def\cT{{\mathcal T}}

\def\cS{{\mathcal S}}

\def\vep{\varepsilon}
\def\ln{\operatorname{ln}}

\title{Cherry flow: physical measures and perturbation theory}

\author{Jiagang Yang}
\date{\today}
\thanks{J.Y. is partially supported by CNPq, FAPERJ, and PRONEX.}
\address{Departamento de Geometria, Instituto de Matem\'atica e Estat\'istica, Universidade
Federal Fluminense, Niter\'oi, Brazil}
\email{yangjg\@@impa.br}

\begin{document}

\begin{abstract} 
In this article we consider Cherry flows on torus which have two singularities: a source and a saddle, 
and no periodic orbits. We show that every Cherry flow admits a unique physical measure, whose basin 
has full volume. This proves a conjecture given by R. Saghin and E. Vargas in~\cite{SV}.
We also show that the perturbation of Cherry flow depends on the divergence
at the saddle: when the divergence is negative, this flow admits a neighborhood, such that any flow 
in this neighborhood belongs to the following three cases: (a) has a saddle connection; (b) a Cherry flow; 
(c) a Morse-Smale flow whose non-wandering set consists two singularities and one periodic sink. In contrary, when the 
divergence is non-negative, this flow can be approximated by non-hyperbolic flow with arbitrarily larger number of 
periodic sinks.

\end{abstract}

\maketitle

\setcounter{tocdepth}{1} \tableofcontents


\section{Introduction}

In this article we consider Cherry flows on torus which have two singularities: a source and a saddle, 
and no periodic orbits and saddle connections. Such an example was given first by Cherry in~\cite{C} (see also~\cite{PM}), and the classification 
of Cherry flows was treated in~\cite{MSMM}. In this article we are interested in the physical measures and 
perturbation theory for Cherry flows. 

An invariant probability $\mu$ of a flow $\phi$ is a {\it physical measure} if its {\it basin}
$$
B(\mu) = \{x\in \mathbb{T}^2: \lim_{t\to +\infty}\frac{1}{t}\int_0^t f(\phi_s(x))ds =\int f d\mu \\
\;\; \text{for every continuous function}\;\; f\}
$$
has positive volume.

We prove the following conjecture given by Saghin and Vargas (~\cite{SV}):

\begin{main}\label{main.A}

Every $\C^\infty$ Cherry flow admits a unique physical measure, whose basin has full volume.

\end{main}
Throughout the paper we consider a Cherry flow $\phi^X$ on torus $\bT^2$ generated by a vector field $X$, with
a source $\tilde{\sigma}^X$ and a quasi-minimal attractor $\Lambda^X$, where $\Lambda^X$ is called {\it quasi-minimal}
if it is transitive, contains a saddle $\sigma^X$, and has neither periodic orbits nor saddle connections. 
We denote by $\cD_X(\sigma^X)$ the divergence at $\sigma^X$. It was observed by~\cite{M,MG,SV} 
that the behavior of physical measures for Cherry flows depends on the divergence at the saddle. And the flows
in general were assumed to be $\C^\infty$, this is because the $\C^\infty$ regularity implies that this flow is $\C^{1+\text{bounded variation}}$ linearizable 
at a neighborhood of $\sigma^X$ (see~\cite{MSMM}[Appendix]).

The existence of physical measures for $\C^\infty$ Cherry flows with non-positive {\it divergence} at the saddle was 
proved in~\cite{SV}[Theorem 1.1]. Theorem~\ref{main.A} is a corollary of the following 
result about the case of positive divergence at the saddle, where the regularity assumption on the flow is removed.

\begin{main}\label{main.B}

Suppose $\phi^X$ is $\C^1$ and $\cD_X(\sigma^X)>0$. Then $\phi^X$ admits a non-trivial physical 
measure supported on $\Lambda^X$, this measure is ergodic, non-hyperbolic, and whose basin has full volume.

\end{main}

An ergodic measure is {\it non-trivial} means that it is not supported on critical elements, i.e., on 
singularities or on periodic orbits. Theorem~\ref{main.B} generalizes the results of~\cite{SV, LP2}. 
Different to most of the dynamical systems with physical measures, where the physical measures in general are 
either hyperbolic (see~\cite{BV,ABV}) or atomic (see~\cite{HY, SSV}), the physical measures in Theorem~\ref{main.B}
are non-hyperbolic and non-trivial. An interesting and important fact is that, the proof of Theorem~\ref{main.B} 
depends on the non-hyperbolicity of this measure.

In~\cite{M,MG}, a different kind of Cherry flows was considered, where $\tilde{\sigma}_0$ is a sink. This sink supports
automatically a physical measure, which is not the case we are interested.

We also consider the perturbation of Cherry flow. It turns out that the perturbation depends also on the 
divergence at the singularity. For a flow $Y$ sufficiently close to $X$, we denote by $\sigma^Y$ the analytic 
continuation of the hyperbolic saddle $\sigma^X$. A flow is called {\it star flow} if the critical elements 
of any $\C^1$ small perturbation are all hyperbolic.

\begin{main}\label{main.C}

Suppose $\phi^X$ is $\C^\infty$ and $\cD_X(\sigma^X)<0$. Then $\phi^X$ is a star flow. Moreover, there is a 
$\C^1$ neighborhood $\cU$ of $\phi^X$, such that every flow $Y\in\cU$ belongs to one of the following three cases: 
\begin{itemize}
\item has a saddle connection;
\item or is a Morse-Smale flow, whose 
non-wandering set consists two singularities and one periodic sink, 
\item or is still a Cherry flow, which admits 
a unique physical measure $\delta_{\sigma^Y}$ with basin $B(\delta_{\sigma^Y})=\bT^2\setminus \tilde{\sigma}^Y$.
\end{itemize}
\end{main}

\begin{main}\label{main.D}

Suppose $\phi^X$ is $\C^1$ and $\cD_X(\sigma^X)\geq 0$. Then $\phi^X$ is $\C^1$ approximated by non-hyperbolic 
flows with arbitrarily large number of periodic sinks.

\end{main}

The techniques of Liao in~\cite{L89} is important for us to remove the regularity assumption on 
the flows, which is mainly used in the proof of Lemma~\ref{l.nonnegative}. This author would like 
to thank Shaobo Gan for his explanation on Liao Theory.

\section{Preliminary}
\subsection{Ergodic measures for Cherry flow}

One can always take a circle $\cS_0$ which does not bound a disk and is everywhere transverse to $X$ e.g., see~\cite{MSMM}
[Proposition 7.1]. The inverse of the Cherry flow $\phi^X$ is a suspension of a continuous circle map $g^X: \cS_0\to \cS_0$,
where $g^X$ is a monotone map and constant on an interval. Such circle maps were well studied in ~\cite{MSMM, GJSTV,Gr,LP1}. Note that $g^X$ has no periodic point, which implies that $g^X$ has 
irrational rotation number $\rho(g^X)$ and is semi-conjugate to the rotation $R_{\rho(g^X)}:x\to x + \rho(g^X) \mod\; 1$. The 
semi-conjugacy $\pi^X$ is continuous, monotone, has degree one and maps orbits of $g^X$ to orbits of $R_{\rho(g^X)}$. Because
the $\pi_X$ pre-image of every point in the circle is either a singular point or a non-trivial connected interval, 
and by the fact that a circle contains at most countable many disjoint open intervals, there are at most countable
many points have non-trivial $\pi^X$ pre-images.
In particular, this set has vanishing Lebesgue measure. Then $\nu^X=(\pi^X)^*\Leb|_{S^1}$ is well defined, and is the unique ergodic measure of $g^X$. For every $\theta\in \cS^1_0$, denote by $\tau_X(\theta)$ 
the first return time of $\theta$ for $\phi^X$. It is well known that $\int \tau_X(\theta) d\nu^X(\theta)<\infty$ if and only if
$$\mu^X=\int_{\cS^1} \int_{t\in[0, \tau_X(\theta))} \delta_{\phi^X_t(\theta)} d\nu^X(\theta)$$ 
is an ergodic measure of $\phi^X$. As a summary, we have that:

\begin{proposition}\label{p.numberofmeasures}

$\phi^X$ has three ergodic measures 
$\delta_{\sigma^X}$, $\delta_{\tilde{\sigma}^X}$ and $\mu^X$ if and only if $\int \tau_X(\theta) d\nu^X(\theta)<\infty$.

\end{proposition}

\begin{theorem}\label{t.ergodicmeasures}[~\cite{SV}[Theorems 1.1, 1.2]]

\begin{itemize}

\item[(a)] Suppose $\phi^X$ is $\C^\infty$ and $\cD_X(\sigma^X)\leq 0$. Then $\delta_{\tilde{\sigma}^X}$ and $\delta_{\sigma^X}$ 
are the only two ergodic measures for $\phi^X$. Consequently, $\delta_{\sigma^X}$ is a physical measure with basin equal 
to $\bT^2\setminus \tilde{\sigma}^X$. 
\item[(b)] Suppose $\phi^X$ is $\C^1$ and $\cD_X(\sigma^X)>0$. Then there exist exactly three ergodic measures for $\phi^X$:
$\delta_{\tilde{\sigma}^X}$, $\delta_{\sigma^X}$ and a third invariant probability measure $\mu^X$ supported on the 
quasi-minimal set $\Lambda^X$.
\end{itemize}

\end{theorem}

\subsection{Ergodic theory for flows}
Throughout this subsection, we consider $\phi^Y$ a flow over surface $M$ which is generated by vector field $Y$, 
and it admits a non-trivial ergodic measure $\mu$. Denote by $\Sing(Y)$ the set of singularities.

The {\it linear Poincar\'e flow $\psi_t$} is defined as following. Denote the normal bundle of $\phi^Y$ over $M$ by
$$\cN_{M}=\cup_{x\in M\setminus \Sing(Y)}\cN_x,$$
where $\cN_x$ is the orthogonal complement of the flow direction $Y(x)$, i.e.,
$$\cN_x=\{v\in T_xM: v\bot Y(x)\}.$$
Denote the orthogonal projection of $T_xM$ to $\cN_x$ by $\pi_x$. Write the tangent flow by $\Phi_t=d \phi^Y_t: TM\to TM$. 
Given $v\in\cN_x, x\in \Lambda\setminus \sigma$, $\psi_t(v)$ is the orthogonal projection of $\Phi_t(v)$ on $\cN_{\phi^Y_t(x)}$ 
along the flow direction, i.e., $$\psi_t(v)=\pi_{\phi^Y_t(x)}(\Phi_t(v))=\Phi_t(v)-\frac{<\Phi_t(v),Y(\phi^Y_t(x))>}{\|Y(\phi^Y_t(x))\|^2}Y(\phi^Y_t(x)),$$
where $<.,.>$ is the inner product on $T_xM$ given by the Riemannian metric. 

\subsubsection{Oseledets Theorem for flows}

The following is the flow version of Oseledets theorem.

\begin{proposition}\label{p.Oseledets}

There exists a real number $\lambda$ such that for $\mu$ almost every $x$:
$$\underset{t\to \pm\infty}{\lim} \frac{1}{t}\log \|\psi_t(v)\|=\lambda\;\; \text{for every non-zero}\;\; v\in \cN_x \setminus \{0\}.$$

\end{proposition}

When the time $t_0>0$ is fixed, $\phi^Y_{t_0}$ is a diffeomorphism and $\mu$ is still an invariant measure of this
diffeomorphism. Suppose $\hat{\mu}$ is an ergodic decomposition of $\mu$ for $\phi^Y_{t_0}$, then 
$$\mu=\int_{t\in [0,t_0)} d (\phi^Y_t)_* \hat{\mu}.$$
We may always choose $t_0$ such that $\mu$ is ergodic for $\phi^Y_{t_0}$. Taking such a time $t_0$ and 
denote $f=\phi^Y_{t_0}$. Now let us state the following relation between Lyapunov exponent of $\mu$ for $\phi^Y_t$ and for $f$, whose
proof is simple and we will not provide.

\begin{theorem}\label{t.exponentsdiffandflow}
The Lyapunov exponents of $\mu$ for $f$ coincide to $\{t_0\lambda\}\cup \{0\}$, and for $\mu$ almost every $x\in M
\setminus \Sing(Y)$, the flow direction $<Y(x)>$ is contained in the subbundle of the Oseledets splitting of $f$ corresponding 
to zero exponent. Moreover, we have $\lim \frac{1}{n}\log \det(Df^n(x))=t_0\lambda$.
\end{theorem}

\begin{remark}\label{r.measurevolume}

The above theorem implies that, for $\mu$ almost every $x$, 
$$\lim_{t\to \infty} \frac{1}{t} \det \Phi_t|_{T_xM}=\lambda.$$

\end{remark}

\begin{definition}

A non-trivial ergodic measure $\mu$ for flow $\phi^Y$ is a {\it hyperbolic measure} if its exponent is non-vanishing. 

\end{definition}

\subsubsection{Divergence}

We need the following version of Liouville Theorem.

\begin{lemma}\label{l.div}

For any $x\in M$ and $t>0$, 
$$\ln \det \Phi_t|_{T_xM}=\int_0^t div_Y(\phi^Y_s(x)) ds.$$

\end{lemma}

The definition of divergence depends on the Riemmanian metric. Because in this paper we only consider
two dimensional torus, for simplicity, in Appendix we provide a short proof of the above lemma with
the assumption that $M=\bT^2$ and the divergence is defined by the flat metric on torus. More precisely,
every small open set of $\bT^2$ can be looked as a subset of $\mathbb{R}^2$, choosing such a local coordinate 
$(x,y)$, in this coordinate we may write $Y=(Y_1, Y_2)$ and $\phi^Y_t=(f_t(x,y),g_t(x,y))$. 
Then for any $x\in \bT^2$, the {\it divergence of $Y$ at $x$} is defined by $div_Y(x)=\frac{\partial Y_1}{\partial x}+\frac{\partial Y_2}{\partial y}$.

\begin{definition}\label{d.div}
Let $\Orb(p)$ be a periodic orbit of $\phi^Y$ with period $\tau_Y(p)$:
\begin{itemize}
\item
the {\it divergence of $\Orb(p)$} is defined by 
$$\cD_Y(\Orb(p))=\frac{1}{\tau_Y(p)}\int_0^{\tau_Y(p)} div_Y(\phi^Y_s(p)) ds;$$
\item 
the {\it divergence of $\mu$} is defined by $\cD_Y(\mu)=\int div_Y(x) d\mu(x)$.

\end{itemize}

\end{definition}
\begin{remark}\label{r.div}
For an atomic measure $\delta_{\sigma}$ supported on a singularity $\sigma\in \Sing(Y)$, 
$\cD_Y(\delta_\sigma)$ coincides to the classical definition of divergence at $\sigma$, $\cD_Y(\sigma)$. And when 
$\mu$ is supported on a periodic orbit $\Orb(p)$, the above two definitions coincide, and by Lemma~\ref{l.div}, 
which equal to $\ln \det \Phi_{\tau_Y(p)}|_{T_pM}=\ln\|\psi_{\tau_Y(p)}|_{\cN_p}\|$.
\end{remark}

As a corollary of Remark~\ref{r.measurevolume} and Lemma~\ref{l.div}, by Birkhoff Ergodic Theorem, we obtain that:

\begin{corollary}\label{c.divformeasures}
$\cD_Y(\mu)$ coincides to the Lyapunov exponent of $\mu$ for the flow $\phi^Y$.

\end{corollary}

The following lemma implies~\cite{SV}[Theorem 1.2]:

\begin{lemma}\label{l.nonpositivedivergence}

There exists an ergodic measure $\mu$ of $Y$ such that $\cD_Y(\mu)\leq 0$.

\end{lemma}

\begin{proof}

Suppose by contradiction that the divergence of every ergodic measure of $\phi^Y$ is positive. Denote by $\cM_Y$ the space of 
invariant measures of $\phi^Y$. Then by Ergodic Decomposition Theorem, there is $a>0$ such that $\cD_Y|_{\cM_Y}>2a>0$. 
Take $\cV$ a small neighborhood of $\cM_Y$ in the probability space over $M$, such that for any $\tilde{\mu} \in \cV$, 
$$\int div_Y(x)d \tilde{\mu}(x)> a.$$ 
Note that the measures contained in $\cV$ are not necessary to be invariant. It is well 
known that for any $x\in M$, there is $t_x>0$ such that $\frac{1}{t}\int_0^{t} \delta_{\phi^Y_s(x)}ds \in \cV$
for any $t\geq t_x$. A short proof is following:
Suppose this is false, then there are $t_n\to \infty$ such that 
$\mu_{n}=\frac{1}{t_n}\int_0^{t_n} \delta_{\phi^Y_s(x)}ds \notin \cV$. Because the probability space
over any compact manifold is compact, we can take a converging point $\mu_0$ of $\{\mu_n\}_{n\in \mathbb{N}}$. 
Then $\mu_0$ is an invariant measure but not contained in $\cV$, a contradiction.  

Let us continue the proof. For $n>0$, denote by $M_n=\{x; t_x\leq n\}$. Because $\bigcup_n M_n=M$, we can
choose $N$ sufficiently large, such that $\Leb(M_N)>0$. Observe that for any
$x\in M_N$ and $n>N$, 
$$\int div_Y(y) d \big(\frac{1}{n}\int_0^n \delta_{\phi^Y_s(x)}ds\big)(y)=\frac{1}{n}\int_0^n div_Y(\delta_{\phi^Y_s(x)}) ds>a.$$

Then by Lemma~\ref{l.div},
$$\det \Phi_n|_{T_xM}=\exp^{\int_0^n div_Y(\delta_{\phi^Y_s(x)})ds}>\exp^{na},$$
which implies that $\lim_n\Leb(\phi_n^Y(M_N))=\infty$, a contradiction to the fact that the volume of $M$ is bounded.

\end{proof}

\subsubsection{Star flow}



Now let us give a general criterion for a two-dimensional flow $Y$ to be a star flow.

\begin{theorem}\label{t.criterionofstarflow}
Suppose all the singularities of $Y$ are hyperbolic, 
and all the ergodic measures of $\phi^Y$ have negative divergence, then $Y$ is a star flow.

\end{theorem}

\begin{example}~\label{ex.figure8}

Consider a 2-dimensional flow $Y_0$ on $\mathbb{R}^2$ with a saddle $\sigma$ of negative divergence, 
and each branch of the unstable manifold of $\sigma$ is connected to a branch of its stable manifold-- 
locally this set looks like a figure `8'. Denote by
$\Gamma$ the set consisting of $\sigma$ and its two saddle connections. It is easy to show that $\Gamma$
is an attractor: for each branch of unstable manifold, we take a transverse section near to $\sigma$ and analyze 
the Poincar\'{e} return maps, which are uniformly contracting. Taking a small contracting neighborhood $U$,
$\delta_\sigma$ is the unique invariant measure supported in $U$. Applying Theorem~\ref{t.criterionofstarflow}, 
$Y_0$ is a star flow in $U$.

\end{example}

\begin{proof}[Proof of Theorem~\ref{t.criterionofstarflow}:]
Because all the singularities of $Y$ are hyperbolic, $\phi^Y$ has only finitely many singularities. 
By the continuation of hyperbolic singularities, the singularities for flows in a small neighborhood are all hyperbolic. 
To prove this theorem, it suffices to show that the periodic orbits of nearby flows are all periodic sinks.

Write $\cM_{Y}$ the space of invariant measures of $\phi^Y$. By the assumption, for any ergodic measure
$\mu\in \cM_{Y}$, $\cD_Y(\mu)<0$. Then by Ergodic Decomposition Theorem, $\cD_Y(.)|_{\cM_{Y}}<0$. Because $\cM_{Y}$ is 
a compact space and $\cD_Y(.)$ is a continuous function, this implies that $\cD_Y(.)|_{\cM_{Y}}<a<0$ for some $a<0$.

For any vector field $Z\in \cU$, suppose $\Orb_Z(p)$ is a periodic orbit of $Z$. Denote $\mu_{\Orb_Z(p)}$ 
the $\phi^Z$ ergodic measure supported on $\Orb_Z(p)$. Then by the upper semi-continuation of the invariant 
measures space, for $\cU$ sufficiently small, $\mu_{\Orb_Z(p)}$ is close to $\cM_{Z}$ in the weak star topology. 
From the definition of divergence for measures,  
$$\cD_Z(\Orb_Z(p))=\cD_Z(\mu_{\Orb_Z(p)})<\frac{a}{2}<0.$$ 
By Remark~\ref{r.div}, $\Orb_Z(p)$ is a periodic sink. The proof is complete.

\end{proof}

As an immediately corollary of Theorem~\ref{t.ergodicmeasures} (a) and Theorem~\ref{t.criterionofstarflow}, we 
show that:

\begin{lemma}\label{l.cherrystarflow}

Suppose $\phi^X$ is a $\C^\infty$ Cherry flow and $\cD_X(\sigma_X)<0$. Then it is a star flow.

\end{lemma}

We also need the following general description on ergodic measures for flows, whose proof depends on Liao Theory, and
is postponed to Section~\ref{s.nonnegative}. The main difficulty in the proof arises from the existence of singularities. 
A similar result for diffeomorphisms can be deduced by considering the corresponding suspension flows 
or using the $\C^1$ version of Pesin Theory for diffeomorphisms in~\cite{Y1}.

\begin{lemma}\label{l.nonnegative}

Suppose $\mu$ is a non-trivial ergodic measure of a $\C^1$ flow. Then $\mu$ has at least one non-negative Lyapunov 
exponent.

\end{lemma}

By considering again the inverse of the flow, we show that:

\begin{corollary}\label{c.nonhyperbolicity}

Every non-trivial ergodic measure of a two-dimensional flow is non-hyperbolic,
i.e., the Lyapunov exponent and the divergence of this measure are both vanishing.

\end{corollary}

\subsection{Perturbation of flows}

We need the following flow version of Franks Lemma of \cite{BGV}[Theorem A.1] to modify
the linear Poincar\'{e} flow along a periodic orbit. Similar statement was obtained by Liao in 
\cite{L79b}[Proposition 3.4].

\begin{theorem}\label{t.Franks}

Given a $\C^1$ neighborhood $\cU$ of $Y$, there is a neighborhood $\cV\subset \cU$ of $Y$ and $\vep>0$ such that for 
any $Z\in \cV$, for any periodic orbit $\Orb_Z(p)$ of $Z$ with period $\tau(p)\geq 1$, any neighborhood $U$ of 
$\Orb_Z(p)$ and any partition of $[0,\tau_Z(p)]$:
$$0=t_0<t_1<\dots <t_l=\tau_Z(p), \;\;\; 1\leq t_{i+1}-t_i \leq 2, i=0,1,\dots, l-1,$$
and any linear isomorphisms $L_i:\cN_{\phi^Z_{t_i}(p)}\to \cN_{\phi^Z_{t_{i+1}}(p)}, i=0,1,\dots, l-1$
satisfying $\|L_i-\psi^Z_{t_{i+1}-t_i}|_{\cN_{\phi_{t_i}(p)}}\|\leq \vep$, there exists $\tilde{Z}\in \cU$ such that
$\psi^{\tilde{Z}}_{t_{i+1}-t_i}|_{\cN_{\phi_{t_i}(p)}}=L_i$ and $\tilde{Z}=Z$ on $(M\setminus U)\cup \Orb_Z(p)$.

\end{theorem}

\begin{remark}\label{r.Franks}
In fact, for any fixed $k\in \mathbb{N}$, one can choose $\cV$ smaller, such that for any $Z\in \cV$, and $p_1,p_2,\dots, p_k$
periodic orbits of $Z$, we can perturb this flow in the above manner near to all the $k$ periodic orbits simultaneously.

\end{remark}

We also need the Closing Lemma:

\begin{theorem}\label{t.closing}

Suppose $x\in M$ is regular point of flow $Y$ and $\omega(x)$ contains regular points. 
Then for any $\C^1$ neighborhood $\cU$ of $Y$ and any neighborhood $U$ of 
$\omega(x)$, there is a  flow $Z\in \cU$ such that $Z=Y|_{M\setminus U}$ and $Z$ has 
a periodic orbit contained in $U$.

\end{theorem}

A deep result of Liao~\cite{L89} shows that the number of periodic sinks is uniformly bounded in a small 
neighborhood of a star flow. The following version of Liao's result which can be obtained directly from the original 
proof, is the motivation of our Theorem~\ref{main.C}. Since it will not be used in this article, we do not 
provide a proof.

For a flow $Y$, denote by $\Sing_1(Y)$ the set of hyperbolic singularities which are contained in non-trivial 
chain recurrent classes, and whose tangent bundle admits a codimension-1 dominated splitting $E^{cs}\oplus E^u_1$
where $E^u_1$ is a one-dimensional unstable bundle. Note that every 
$\sigma\in \Sing_1(Y)$ admits a one-dimensional strong unstable manifold. We also write $S(Y)$ 
the number of periodic sinks of $\phi^Y$ and $K(Y)=\#(\Sing_1(Y))$.

\begin{theorem}\label{t.Liao}
Suppose $Y$ is a star flow. Then $S(Y)$ is finite, and there is a $\C^1$ neighborhood $\cU$ of $Y$, such that for any
$Z\in \cU$, 
$$S(Z)\leq S(Y)+2K(Y).$$ 
Moreover, the basin of every `new' periodic sink of $Z$ intersects one branch of the 
one-dimensional strong unstable manifold of a singularity of $Z$ which is the analytic continuation of a singularity 
$\sigma\in \Sing_1(Y)$.

\end{theorem}

\begin{example}
Applying the above theorem on the star flow $Y_0$ in Example~\ref{ex.figure8}, there is an open neighborhood $\cU$ of $Y_0$, 
such that any flow $Z\in\cU$ has at most two periodic orbits/sinks, and a saddle connection is broken when one periodic sink
appears.
\end{example}








\section{Proof of Lemma~\ref{l.nonnegative}\label{s.nonnegative}}
Throughout this section we suppose $\phi^Y$ is a flow generated by vector field $Y$ over manifold $M^d$, $\mu$ is a
non-trivial ergodic measure of $\phi^Y$ with all Lyapunov exponents negative. We prove Lemma~\ref{l.nonnegative} 
by showing the contradiction. In fact, we make use of scaled linear Poincar\'{e} flow in Liao Theory to guarantee that there are infinite 
number of distinct periodic sinks, and each basin contains a uniform size of ball, which is a contradiction.

We need the following version of~\cite{Y1}[Theorem 5.1], where the `$F$ bundle' of the original statement is taken to be 
empty here. For completeness, we give a proof in Subsection~\ref{ss.shadowing}.

\begin{theorem}\label{t.shadowing}
There is a compact, positive $\mu$ measure subset $\Lambda\subset \supp(\mu)\setminus \Sing(Y)$ 
satisfying: 
for any $\vep>0$, there are $L,L^{'}, \delta>0$ such that for any $x$ and $\phi^Y_T(x)\in \Lambda$ 
with $T>L^{'}$ and $d(x,\phi^Y_T(x))<\delta$, there exists a point $p\in M$ and a $\C^1$ strictly 
increasing function $\theta:[0,T] \rightarrow \mathbb{R}$ such that:
\begin{itemize}

\item[(a)] $\theta(0)=0$ and $1-\vep <\theta^{'}(t)<1+\vep$;

\item[(b)] $p$ is a periodic sink of $\phi^Y$ with period $\tau(p)=\theta(T)$;

\item[(c)] $d(\phi^Y(x),\phi^Y_{\theta(t)}(p))\leq \vep\|Y(\phi^Y_t(x))\|, t\in [0,T]$;

\item[(d)] $\Orb_Y(p)$ has uniform size of stable manifold at $p$.

\end{itemize}
\end{theorem}

\begin{proof}[Proof of Lemma~\ref{l.nonnegative}:]
Fix $\Lambda$, $\vep, \delta, L, L^{'}>0$ as in Theorem~\ref{t.shadowing}. By Birkhoff Ergodic Theorem,
there is $x\in \Lambda$ and $L^{'}<t_1<t_2<\dots$ such that
\begin{itemize}
\item[(i)] $\phi^Y_{t_i}(x)\in \Lambda$ for any $i>0$ and $\phi^Y_{t_i}(x)\to x$;

\item[(ii)] $t_{i+1}(1-\vep)>t_i(1+\vep)$.
\end{itemize}

Applying Theorem~\ref{t.shadowing}, each pseudo orbit $\big(\phi_t^Y(x)\big)_{t\in [0,t_i]}$ is 
shadowed by a periodic sink $p_i$ with period $(1-\vep)t_i<\tau(p_i)<(1+\vep)t_i$. 

By (ii) above, all the periods are different, which implies that all these periodic sinks are distinct. By (d) 
of Theorem~\ref{t.shadowing}, each periodic sink $\Orb(p_i)$ has uniform size of stable manifold at $p_i$ 
for every $i$, which is a contradiction.

\end{proof}

\subsection{Proof of Theorem~\ref{t.shadowing} \label{ss.shadowing}}

In this subsection, we provide a proof of Theorem~\ref{t.shadowing}. In the proof we need another flow $\psi^*_t: \cN_M \rightarrow \cN_M$, which 
is called {\it scaled linear Poincar\'e flow}:
$$\psi^*_t(v)=\frac{\|Y(x)\|}{\|Y(\phi_t(x))\|}\psi_t(v)=\frac{\psi_t(v)}{\|\Phi_t|_{<Y(x)>}\|},$$

\begin{lemma}\label{l.twococycles}

$\psi^*_t$ is a bounded cocycle on $\cN_M$ in the following sense: for any $\tau>0$, there is
$C_\tau>0$ such that for any $t\in[-\tau,\tau]$,
$$\|\psi^*_t\|\leq C_\tau.$$
Moreover, the two cocycles
$\psi_t$ and $\psi^*_t$ have the same Lyapunov exponent.
\end{lemma}

\begin{proof}

The uniform boundness comes from the boundness of $\Phi_t$ for $t\in [-\tau,\tau]$.

For $\mu$ almost every $x$, we have $\lim \frac{\log\|\Phi_t(Y(x))\|}{t}=0$. This implies that both
$\psi_t$ and $\psi^*_t$ have the same Lyapunov exponents and Oseledets splitting.

\end{proof}

\subsubsection{Pesin block}

\begin{definition}\label{d.quasihyperbolic}
The orbit arc $\phi_{[0,T]}(x)$ is called {\it $(\lambda,T_0)^*$-quasi contracting} with respect to the
scaled linear Poincar\'e flow $\psi^*_t$ if there exists $0<\lambda<1$ and a partition
$$0=t_0<t_1<\dots<t_l=T,\;\;\; \text{where}\;\; t_{i+1}-t_i\in [T_0, 2T_0]$$
such that for $k=0,1,\dots, l-1$, we have
$$\prod_{i=0}^{k-1}\|\psi^*_{t_{i+1}-t_i}|_{\cN_{\phi_{t_i}(x)}}\|\leq \lambda^k$$

\end{definition}

Now we state a $\C^1$ version of Pesin block.

\begin{lemma}\label{l.goodset}
There are $L^{'}, \eta, T_0>0$ and a positive $\mu$ measure subset $\Lambda\setminus \Sing(X)$, such that 
for any $x, \phi^Y_T(x)\in \Lambda$ with $T>L^{'}$, $\phi^Y_{[0,T]}(x)$ is $(\eta,T_0)^*$-quasi contracting with 
respect to the scaled linear Poincar\'e flow $\psi^*_t$.
\end{lemma}

\begin{proof}
Taking $1\leq t_0 <2$ such that $\mu$ is an ergodic measure for $f=\phi^Y_{t_0}$. For simplicity, 
we assume $t_0=1$.

By subadditive ergodic theorem, there is $a<0$ such that
$$\lim_{t\to \infty}\frac{1}{t}\int \log \|\psi_t^*\|d\mu<a.$$

Take $N_0$ sufficiently large, such that $\frac{1}{N_0}\int \log \|\psi_{N_0}^*\|d \mu<a.$ 
Consider the ergodic decomposition of $\mu$ for $f^{N_0}=\phi^Y_{N_0}$:
$$\mu=\frac{1}{k_0}\big(\mu_1+\dots+\mu_{k_0}\big).$$
By changing the order, we suppose that:
$$\frac{1}{N_0}\int \log \|\psi_{N_0}^*\|d \mu_1<a.$$

Note that $\mu_1$ is an $f^{N_0}$ ergodic measure. Applying Birkhoff ergodic theorem, for $\mu_1$ almost every $x$, 
$$\lim_m \frac{1}{mN_0}\sum_{i=0}^{m-1}\log\|\psi_{N_0}^*|_{\cN_{f^{iN_0}(x)}}\|<a.$$ 
There is $n_x>0$ such that for any $m>n_x$, 
$$\frac{1}{mN_0}\sum_{i=0}^{m-1}\log\|\psi_{N_0}^*|_{\cN_{f^{iN_0}(x)}}\|<a.$$

Choose $N_1$ such that the set $\Lambda^{'}=\{x; n(x)<N_1\}$ has positive $\mu_1$ measure. Let 
$\Lambda\subset \Lambda^{'}$ be a compact subset with positive $\mu_1$ measure. It follows
immediately that $\mu(\Lambda)>0$. By Lemma~\ref{l.twococycles}, let
$$K=\sup_{|t|\leq N_0,y\in \Lambda}\{\|\psi_t^*|_{\cN_{y}}\|\}.$$
Choose $N_2$ sufficiently large and $b<0$ such that 
$$\frac{N_2+N_0}{N_0}a+3K<b<0.$$

We claim that for any sequence $n_1<n_2\cdots<n_l$ satisfying $N_2\leq n_{i+1}-n_i\leq N_2+N_0$ for $0\leq i \leq l-1$
and $x\in \Lambda$: 
$$\frac{1}{l}\sum_{i=0}^{l-1}\log \|\psi^*_{n_{i+1}-n_i}|_{\cN_{f^{n_i}(x)}}\|<b.$$

Let continue the proof. Choose $L^{'}>0$ be sufficiently large, such that for any $n>L^{'}$, 
there always exists a sequence $n_1<n_2\cdots<n_l$ satisfying $N_2\leq n_{i+1}-n_i\leq N_2+N_0$ for each $0\leq i \leq l-1$.
Then by the above claim, we conclude the proof of this lemma.

It remains to prove the claim.
For each $0\leq i<l$, denote 
$$k_i=[\frac{n_{i+1}}{N_0}]-[\frac{n_{i}}{N_0}]-1, n^{'}_i=([\frac{n_i}{N_0}]+1)N_0\;\; \text{and}\;\; n^*_{i+1}=[\frac{n_{i+1}}{N_0}]N_0.$$
Then $n^*_{i}\leq n_i \leq n^{'}_{i}$, $n^*_{i+1}-n^{'}_i=k_i N_0$ and
$$\psi^*_{n_{i+1}-n_{i}}|_{\cN_{f^{n_i}(x)}}= \psi^*_{n_{i+1}-n^*_{i+1}}|_{\cN_{f^{n^*_{i+1}}(x)}}\circ \psi^*_{k_iN_0}|_{\cN_{f^{n_i^{'}}(x)}}\circ\psi^*_{n^{'}_i-n_i}|_{\cN_{f^{n_i}(x)}}$$
Observe that $n_{i+1}-n^*_{i+1}\leq N_0$ and $n^{'}_i-n_i\leq N_0$, which imply in particular that 
\begin{equation*}
\begin{split}
\log\|\psi^*_{n_{i+1}-n_i}|_{\cN_{f^{n_i}(x)}}\|& \leq 2K+\log \|\psi^*_{k_iN_0}|_{\cN_{f^{n_i^{'}}(x)}}\| \\
& \leq 3K+ \log\|\psi^*_{n^{'}_{i+1}-n^{'}_i}|_{\cN_{f^{n_i^{'}}(x)}}\|.\\
\end{split}
\end{equation*}

Because $n^{'}_0=0$, we have that
\begin{equation*}
\begin{split}
\frac{1}{l}\sum_{i=0}^{l-1}\log\|\psi^*_{n_{i+1}-n_i}|_{\cN_{f^{n_i}(x)}}\|& \leq \frac{1}{l}\sum_{i=0}^{l}\log\|\psi^*_{n^{'}_{i+1}-n^{'}_i}|_{\cN_{f^{n_i^{'}}(x)}}\|+3K \\
& \leq \frac{1}{l}\sum_{j=0}^{\frac{n^{'}_{l+1}}{N_0}} \log\|\psi^*_{N_0}|_{\cN_{f^{jN_0}(x)}}\|+3K \\
&\leq \frac{n^{'}_{l+1}}{lN_0} a+3K  \leq \frac{N_2+N_0}{N_0} a+3K \leq b\\
\end{split}
\end{equation*}

\end{proof}

\subsubsection{Liao's shadowing lemma for scaled linear Poincar\'e flow}

In this subsection we introduce the Liao's shadowing lemma for scaled linear Poincar\'e flow of~\cite{L89}. 
The idea of the proof is explained:

\begin{theorem}\label{t.liaoclosinglemma}

Given a compact set $\Lambda\cap \Sing(Y)=\emptyset$, and 
$\eta>0, T_0 > 0$, for any $\vep > 0$ there exist $\delta> 0$ and $L>0$, such that for any $(\eta,T_0)^*$-quasi 
contracting orbit arc $\phi_{[0,T]}(x)$ with respect to the scaled linear Poincar\'e flow $\psi_t^*$ 
which satisfies $x, \phi^Y_T(x)\in \Lambda$ and $d(x,\phi_T(x))\leq \delta$,
there exists a point $p\in M$ and a $\C^1$ strictly increasing function $\theta:[0,T] \rightarrow \mathbb{R}$ 
such that
\begin{itemize}

\item[(a)] $\theta(0)=0$ and $1-\vep <\theta^{'}(t)<1+\vep$;

\item[(b)] $p$ is a periodic sink with $\tau(p)=\theta(T)$;

\item[(c)] $d(\phi^Y_t(x),\phi^Y_{\theta(t)}(p))\leq \vep\|Y(\phi^Y_t(x))\|, t\in [0,T]$;

\item[(d)] $\Orb(p)$ has uniform size of stable manifold at $p$.
\end{itemize}
\end{theorem}

\begin{proof}

For simplicity, we assume $M$ is an open set in $R^d$, which implies that for every regular point $x\in M$, $\cN_x$ is a $d-1$-dimensional hyperplane. For $\beta>0$, we denote by $\cN_x(\beta)$ the ball contained in $\cN_x$ with radius $\beta$.

The first step is to translate the problem for flow into the shadowing lemma for a sequence of maps:
by~\cite{GY}[Lemmas 2.2, 2.3], there is $\beta$ depending on $T_0$ such that the holonomy map induced by $\phi^Y$ is 
well defined between 
$$\cP_{x,\phi^Y_t(x)}: \cN_x(\beta\|Y(x)\|)\rightarrow \cN_{\phi^Y_t(x)}\;\;\; \text{for any}\;\;\; t\in [T_0,2T_0],$$
which is conjugate to the following map:
$$\cP^*_{x,\phi^Y_t(x)}, \cN_x(\beta)\rightarrow \cN_{\phi^Y_t(x)}:  \cP^*_{x,\phi^Y_t(x)}(a)=\frac{\cP_{x,\phi^Y_t(x)}(a\|Y(x)\|)}{\|Y(\phi^Y_t(x))\|}.$$
The tangent map of $\cP^*_{x,\phi^Y_t(x)}$ is uniformly continuous.

Take the sequence of times $0=t_0<t_1<\dots<t_l=T$ in the definition of quasi contracting orbit. 
We consider the local diffeomorphisms induced by the holonomy maps:
$$\cT_i= \cP^*_{\phi^Y_{t_{i-1}}(x),\phi^Y_{t_{i}}(x)}\;\;\; \text{for}\;\;\; 1\leq i\leq l-1;$$
and $ \cT_{l}: \cN_{\phi^Y_{t_{l-1}}(x)}(\beta)\rightarrow \cN_x$. Because $d(x,\phi^Y_T(x))$ is sufficiently small, 
the holonomy map $\cT_l$ is well defined. For each $1\leq i \leq l$, we have
$$T\cT_i(0)=\psi^*_{t_{i}-t_{i-1}}|_{\cN_{\phi^{Y}_{t_{i-1}}(x)}}.$$
Then, by the definition of quasi contracting, there is $0<\lambda<1$ such that for $k=1,\dots, l$, we have
$$\prod_{i=1}^{k}\|\cT_i\|\leq \lambda^k.$$

Now we apply the version of Liao's shadowing lemma on discrete quasi hyperbolic maps (\cite{G}[Theorem 1.1]), 
see also~\cite{L79a, L85, L96} and the explanation below.
On the $d(x,\phi^Y_{T}(x))$-pseudo orbit
$\{0_x,0_{\phi^Y_{t_1}(x)},\dots, 0_{\phi^Y_{t_{l-1}}(x)}\}$, there is $L>0$ and a 
periodic point $p\in \cN_x$ for the maps $\{\cT_0,\cdots, \cT_{l-1}\}$, 
whose orbit $Ld(x,\phi^Y_{T}(x))$-shadows the pseudo orbit.

Using the version of flow tubular theorem~\cite{GY}[Lemmas 2.2], one can prove (a), (b) and (c) above. 
Now let us focus on the proof of (d). We follow the proof of~\cite{G}. By Lemma 3.1 of~\cite{G}, there is a sequence of positive numbers
$\{c_i\}_{i=1}^{l}$ called {\it well adapted} such that:
\begin{itemize}

\item[(i)] $g_j=\prod_{j=1}^k c_j\leq 1 $, $k=1,\dots, l-1$ and $g_{l}=\prod_{j=1}^{l} c_j= 1 $;

\item[(ii)] denote $\tilde{\cT}_j(a)= g_j^{-1}\cT_j(g_{j-1}a)$ for $a\in \mathbb{R}^{d-1}$, then
$$\|T\tilde{\cT}_j(0)|_{E_j}\|\leq \lambda.$$

\end{itemize}

It was shown in~\cite{G} (p. 631) that the well adapted sequence is uniformly bounded from above and from zero, and the
sequence of contracting maps $\{\tilde{\cT}_j\}_{j=1}^{l}$ are Lipschtez maps (with uniform Lipschtez constant). 

Denote by 
$$\Psi_k= \cT_{k} \cdots \circ \cT_1 \;\; \text{and}\;\; \tilde{\Psi}_k=\tilde{\cT}_{k}\cdots \circ \tilde{\cT}_1.$$ 
It is easy to see that 
\begin{equation}\label{e.conjugation}
\begin{split}
\Psi_k=g_k \tilde{\Psi}_k \;\; & \text{and}  \;\;\Psi_l=\tilde{\Psi}_l. \\
\end{split}
\end{equation}

Observe that $\{0_{x},0_{\phi^Y_{t_1}}(x),\cdots, 0_{\phi^Y_{t_{l-1}}(x)}\}$ is still a $d(x,\phi^Y_{T}(x))$-pseudo orbit of
sequence of uniformly contracting diffeomorphisms $\{\tilde{\cT}_j\}_{j=0}^{l-1}$. By the standard shadowing lemma for hyperbolic
diffeomorphisms (see~\cite{G}[Lemma 2.1]), there is $L>0$ such that the pseudo orbit is $Ld(x,\phi^Y_{T}(x))$-shadowed by periodic orbit $\{\tilde{p},\tilde{p}_1,\dots, \tilde{p}_{l-1}\}$, and this periodic orbit has uniform size
of stable manifold. By \eqref{e.conjugation}, $\{g_k \tilde{p}_k\}_{k=1}^{l}$ is a pseudo orbit for the 
original sequence of maps $\{\cT_k\}_{k=0}^{l-1}$. Observe that $g_l=1$, which implies that $p_0$ and $\tilde{p}_0$ coincide 
and have the same stable manifold. Hence, $p_0$ has uniform size of stable manifold. We conclude the proof of (d).

\end{proof}
\subsubsection{Proof of Theorem~\ref{t.shadowing}}

It is easy to see that Theorem~\ref{t.shadowing} is a direct corollary of Lemma~\ref{l.goodset}
and Theorem~\ref{t.liaoclosinglemma}.

\section{Proof of Theorem~\ref{main.B}}

\begin{proof}[Proof of Theorem~\ref{main.B}]:
By Theorem~\ref{t.ergodicmeasures}(b), $\Lambda^X$ admits two ergodic measures: $\delta_{\sigma^X}$ and $\mu^X$. 
By Corollaries~\ref{c.nonhyperbolicity} and~\ref{c.divformeasures}, $\mu^X$ is non-hyperbolic and $\cD_{X}(\mu^X)=0$. 

For every $x\in M\setminus \tilde{\sigma}^X$, its forward orbit converges to the quasi-minimal set. In particular, any accumulated point
of $\frac{1}{t}\int_0^{t} \delta_{\phi^X_s(x)} ds$ can be written as $a\delta_{\sigma^X}+(1-a)\mu^X$ for some $0\leq a \leq 1$.

For $0\leq b \leq 1$, denote $K_b=\{a\delta_{\sigma^X}+(1-a)\mu^X; b\leq a \leq 1\}$ and 
$$\Gamma_b=\{x\in \bT^2;\; \text{there is}\; t_i\to \infty\; \text{such that}\; \lim_i \frac{1}{t_i}\int_{0}^{t_i} \delta_{\phi^X_s(x)} ds \in K_b\}.$$
We claim that $\Leb(\Gamma_b)=0$ for every $b>0$. It is easy to see that Theorem~\ref{main.B} follows from this claim
immediately. Now let us prove this claim.
\begin{proof}
Because $\cD_X(\sigma^X)>0$ and $\cD_X(\mu^X)=0$, there is $c>0$ such that $\cD_X|_{K_b}>2c$. Taking $\cV$ a small 
neighborhood of $K_b$ in the probability space of $\bT^2$, such that for any measure $\mu\in \cV$ which is possibly 
not invariant, we have $\int div_X(x)d\mu(x)>c$.

Denote by 
$$\Gamma_{b,n}=\{x;\; \text{there is} \; t_x\in [n,n+1)\; \text{such that}\; \frac{1}{t_x}\int_{0}^{t_x}\delta_{\phi^X_{s}(x)} ds\in \cV\}.$$
Then for every $x\in \Gamma_{b,n}$, we have 
$$\int div_X(y) d\big(\frac{1}{t_x}\int_{0}^{t_x}\delta_{\phi^X_{s}(x)}ds\big)(y) =\frac{1}{t_x}\int_{0}^{t_x} div(\phi^X_s(x)) ds >c ,$$
which implies by Lemma~\ref{l.div} that
$$\det \Phi_{t_x}|_{T_x\bT^2}=\exp^{t_x\int_{0}^{t_x} div_X(\phi^X_s(x)) ds}>\exp^{ct_x}.$$
Denote by $C=\max_{x\in \bT^2, t\in [0,1)}\{\det \Phi_t|_{T_x\bT^2}\}$. Then 
for any $x\in \Gamma_{b,n}$, 
$$\det \Phi_{n}|_{T_x\bT^2}>\frac{1}{C}\exp^{cn},$$ 
which implies that $\Leb(\Gamma_{b,n})\leq C\exp^{-cn}$. 

Because $\Gamma_b\subset \bigcap_n \bigcup_{i>n} \Gamma_{b,n}$, the above argument shows that $\Leb(\Gamma_b)=0$. 
The proof of the claim is complete.
\end{proof}
\end{proof}

\section{Proof of Theorems~\ref{main.C} and~\ref{main.D}}
\subsection{Perturbation of Cherry flow\label{ss.perturbation}}

Recall that $\cS_0$ is a circle transverse to $X$. There is a $\C^1$ neighborhood $\cU_0$ of $X$
such that for any flow $Y\in\cU_0$, $\cS_0$ is still transverse to $Y$. Moreover, the inverse of $\phi^Y$ 
can be looked as a suspension of a continuous circle map $g^Y$, where $g^Y$ is a monotone circle map 
which is constant on an interval $I^Y$. The intersection between 
a periodic orbit of $\phi^Y$ with $\cS_0$ corresponds to a periodic orbit of $g^Y$. In contrary, 
a periodic orbit of $g^Y$ which does not intersect with $I^Y$ corresponds to a periodic orbit of 
$\phi^Y$. Denote by $\Per(Y)$ and $\Per(g^Y)$ the set of periodic orbits for $Y$ and $g^Y$ respectively.
A segment $I=[p,q]\subset \cS_0$ is called a {\it $g^Y$-periodic segment} if $p,q$ both are periodic 
points of $g^Y$. The proof of the following lemma is the same as for homeomorphisms over circle.

\begin{lemma}\label{l.transversehomeo}

Suppose $g^Y$ has a periodic point with period $k$. Then its rotation number $\rho(g^Y)$ is rational, all the periodic 
points have the same period, and every periodic segment is fixed by $(g^Y)^k$. Moreover, 
for any $\theta\in \cS_0$, $\omega(\theta)$ is a periodic orbit of $g^Y$.

\end{lemma}

\begin{corollary}\label{c.periodicpoint} 
$\#\Per(Y)\geq \#\Per(g^Y)-1$. More precisely, suppose $Y$ contains $n$ periodic sources, 
then it has at least $n$ periodic orbits which are not periodic sources, and $\#\Per(Y)\geq 2n$; in contrary,
when $Y$ contains $n$ periodic sinks, then it has at least $n-1$ periodic orbits which are not periodic sinks,
and $\#\Per(Y)\geq 2n-1$.
\end{corollary}

\begin{proof}
Observe that $g^Y(I^Y)$ is a point, $I^Y$ intersects at most one periodic orbit, and in this case, 
$I^Y$ belongs to the stable set of this periodic orbit. Hence, the periodic orbit which intersects 
with $I^Y$ is not a source for $g^Y$. 

When $Y$ has $n$ periodic sinks: $\Orb_Y(p_1),\dots, \Orb_Y(p_n)$, suppose
$p_1,\dots, p_n\in \cS_0$. Then for every $1\leq i \leq n$, $p_i$ is an isolated source for $g^Y$. 
For each $1\leq i \leq n$, take $q_i$ the nearest periodic point of $g^Y$ to $p_i$ in the counterclockwise 
direction, then $q_i$ has a half neighborhood which is topologically contracting by $g^Y$. There are at least 
$\#\Per(g^Y)-1$ distinct such periodic orbits of $g^Y$ which do not 
intersect $I^Y$, these periodic orbits correspond to intersections of periodic orbits of $Y$ with $\cS_0$. 
Therefore, there are at least $n-1$ $g^Y$-periodic orbits of $\{\Orb_{g^{Y}}(q_1),\dots, \Orb_{g^Y}(q_n)\}$ 
can be lifted to periodic orbits of $Y$ which are not periodic sinks. 
As a conclusion, $\phi^Y$ has at least $2n-1$ periodic orbits.

When $Y$ has $n$ periodic sources: $\Orb_Y(p_1),\dots, \Orb_Y(p_n)$, suppose
$p_1,\dots, p_n\in \cS_0$. Then for every $1\leq i \leq n$, $p_i$ is an isolated sink for $g^Y$. 
We first consider the case that $I^Y$ does not intersect a periodic orbit of $g^Y$. Taking $q_i$ the nearest 
periodic point of $g^Y$ to $p_i$ in the counterclockwise direction, then $q_i$ has at least a half 
neighborhood which is expanding by $g^Y$. As observed above, $\{\Orb_{g^{Y}}(q_1),\dots, \Orb_{g^Y}(q_n)\}$ 
can be lifted to periodic orbits of $Y$ which are not periodic sources. Hence, $Y$ has at least $2n$ periodic orbits.

In the second case, we suppose $p_0\in I^Y$ is a periodic point of $g^Y$. We may further suppose that 
$I\setminus p_0$ contains a segment in the counterclockwise direction, the proof of other case is similar. 
Taking $q_i$ $i=0,1,\dots,n$ the nearest periodic point of $g^Y$ in the counterclockwise direction, Then each
$q_i$ has at least a half neighborhood is expanding by $(g^Y)^k$. There are at least $n$ $g^Y$-periodic orbits of 
$\{q_0,\dots, q_n\}$ can be lift to periodic orbits of $Y$, which are not periodic sources. This implies that
$Y$ has at least $2n$ periodic orbits.

The proof of this corollary is finished.
\end{proof}

\subsection{Proof of Theorems~\ref{main.C} and~\ref{main.D}}
\begin{proof}[Proof of Theorem~\ref{main.C}:]

Fix $\cU_0$ the neighborhood of $X$ given in subsection~\ref{ss.perturbation}.
By Lemma~\ref{l.cherrystarflow}, $\phi_t$ is a star flow. There is a $\C^1$ neighborhood $\cU\subset \cU_0$ 
of $X$ such that for any flow $Y\in \cU$, $Y$ belongs to the following three cases:
\begin{itemize}
\item[(a)] has a saddle connection; 
\item[(b)] has no periodic orbits, 
\item[(c)] has periodic orbits, and all the periodic orbits are periodic sinks. 

\end{itemize}
From now on, we assume that $Y$ has no saddle connection.

Recall the general definition of Cherry flow in~\cite{MSMM}:
\begin{definition}
A flow on $\bT^2$ is called a {\it Cherry flow} if it has:
\begin{itemize}
\item[(i)] a finite number of singularities all of which are hyperbolic;
\item[(ii)] no periodic orbits;
\item[(iii)] no saddle-connections;
\item[(iv)] sinks. 
\end{itemize}
\end{definition}

When $\phi^{Y}$ contains no periodic orbit, by the above definition of Cherry flow, 
$Y$ is automatically a $\C^1$ Cherry flow. Denote by $\Lambda^Y$ the quasi-minimal attractor and 
$\sigma^Y$ the singularity contained in $\Lambda^Y$. By Proposition~\ref{p.numberofmeasures}, 
$\Lambda^Y$ admits at most two ergodic measures. 
We claim that $\Lambda^Y$ admits a unique ergodic measure $\delta_{\sigma^Y}$. 
\begin{proof}
Suppose that $\Lambda^Y$ admits another non-trivial ergodic measure $\mu^Y$ of $\phi^Y$.
Because $\Lambda^X$ only supports a unique invariant measure $\delta_{\sigma_0}$ of $\phi^X$ 
with negative divergence, by the upper-semi continuation of the invariant measures spaces, 
we have that $\int div_Y(x)d \mu^Y(x)<0$. This contradicts to Corollary~\ref{c.nonhyperbolicity}.

\end{proof}

The above claim implies that $\delta_{\sigma^Y}$ is a physical measure, whose basin coincides to 
$\bT^2\setminus \tilde{\sigma}^Y$ where $\tilde{\sigma}^Y$ denotes the unique source of $Y$.

Now suppose $\phi^Y$ contains a periodic sink $p$. Let us show that $\Orb_Y(p)$ is the unique periodic orbit of 
$\phi^Y$. Suppose there is a different periodic sink $\Orb_Y(q)$.
By Corollary~\ref{c.periodicpoint}, $\phi^Y$ has a periodic orbit which is not a periodic sink. This contradicts to the 
fact that every periodic orbit of $\phi^Y$ is a periodic sink. Hence, $\phi^Y$ contains only one periodic sink.

It remains to show that $\phi^Y$ is a Morse-Smale flow. We are going to prove that every $x\in \bT^2\setminus (\tilde{\sigma}^Y\cup\sigma^Y\cup W^s(\sigma^Y))$ belongs to the attracting basin of the periodic sink $\Orb_Y(p)$.
Suppose this is false, then there is $x\in \bT^2\setminus (\tilde{\sigma}^Y\cup\sigma^Y\cup W^s(\sigma^Y))$ such that 
$\omega(x)$ is not a critical element. Take a smaller neighborhood $V$ of $\omega(x)$ such that 
$V\cap \Orb_Y(p)=\emptyset$. Then by Closing Lemma, there is another flow $Z\in \cU$ such that $Z|_{V^c}=Y$ and
$\phi^Z$ has a periodic orbit $\Orb_Z(q)\subset V$. Then $\phi^Z$ has two periodic orbits, a contradiction
to the above discussion that $\phi^Z$ has at most one periodic orbit. The proof is complete.

\end{proof}

\begin{proof}[Proof of Theorem~\ref{main.D}:]

We first prove that $\phi_t$ can be approximated by flows with arbitrarily large number of periodic sinks. Suppose by 
contradiction that there exist $l>0$ and $\cU\subset \cU_0$ a neighborhood of $X$ such that every flow $Y\in \cU$ has
at most $l$ periodic sinks. We claim that $l\geq 1$. 
\begin{proof}
Suppose by contradiction that $l=0$.
By Closing lemma, there is a flow $Y$ arbitrarily close to $X$ such that $Y$ has at least one periodic orbit
$\Orb_Y(p)$ which is not a periodic sink. After a small perturbation at the singularity $\sigma^Y$, we may 
always assume that $\cD_Y(\sigma^Y)>0$. When $\cD_Y(\Orb_Y(p))=0$, which means this periodic orbit is non-hyperbolic,
we can apply the Franks Lemma (Theorem~\ref{t.Franks}) to show that, there is flow $Z$ arbitrarily close to $Y$ such that $\Orb_Z(p)$ 
is a periodic sink. This contradicts our previous assumption. 

From now on, we assume that all the periodic orbits of $Y$ are sources. 
By Corollary~\ref{c.periodicpoint}, $Y$ has at least another periodic orbit $\Orb_Y(q)$ which
is not a periodic source. This contradicts to the above assumption that all the periodic orbits
of $Y$ should be periodic sources.

\end{proof}

Take a flow $Y$ which is arbitrarily close to $X$ such that $\phi^Y$ contains $l$ periodic sinks 
$\Orb_Y(p_1),\dots,\Orb_Y(p_l)$. By Corollary~\ref{c.periodicpoint}, $Y$ contains at least $l-1$ 
periodic orbits which are not periodic sinks. Replacing by arbitrary small perturbation, we suppose 
the $l-1$ periodic orbits $\Orb_Y(q_1),\dots, \Orb_Y(q_{l-1})$ are all periodic sources. Because 
all the invariant measures of $X$ have non-negative divergence, by the upper semi-continuation 
of the invariant measures space, there is $\vep$ close to zero such that for any $1\leq i\leq l$, 
$\cD_Y(\Orb_Y(p_i))\geq -\vep$. Applying Franks Lemma as observed in Remark~\ref{r.Franks}, 
there is a flow $Z$ which is close to $Y$ such that 
\begin{itemize}
\item $\{\Orb_Z(p_i);i=1,\dots, l\}$ are all periodic sources with divergence smaller than $\vep$,
\item $Z=Y$ outside a small neighborhood of $\cup_{i=1,\dots, l} \Orb_Y(p_i)$. 
\end{itemize}
In particular, $Z$ has $2l-1$ periodic sources. By Corollary~\ref{c.periodicpoint}, $Z$ has at least
another $2l-1$ periodic orbits $\Orb_Z(a_{1}),\dots, \Orb_Z(a_{2l-1})$ which are not periodic sources. 
Because $\{\Orb_Z(p_i)\}_{i=1}^l$ all have small divergence, using again the Franks Lemma, we obtain a flow $\hat{Z}$ such that 
$$\Orb_{\hat{Z}}(p_1),\dots, \Orb_{\hat{Z}}(p_l)),\Orb_{\hat{Z}}(a_1),\dots, \Orb_{\hat{Z}}(a_{2l-1})$$
are all periodic sinks. Hence $\hat{Z}$ has at least $3l-1>l$ periodic sinks, a contradiction to the assumption
that $\hat{Z}$ has at most $l$ periodic sinks. This contradiction shows that $X$ can be approached by
a flow $Y$ with arbitrarily large number of periodic sinks.

For any $l>0$, take $Y$ sufficiently close to $X$ such that $Y$ has $l$ periodic sinks $\{\Orb_Y(p_1),\dots, \Orb_Y(p_l)\}$.
As explained above, all these periodic sinks have negative divergence close to zero. Applying Franks Lemma on $\Orb_Y(p_1)$, 
there is a flow $Z$ arbitrarily close to $Y$ such that 
\begin{itemize}
\item $\Orb_Z(p_1)$ is not-hyperbolic, 
\item $\Orb_Z(p_2),\dots, \Orb_Z(p_l)$ are still periodic sinks of $Z$. 
\end{itemize}
The proof is complete.
\end{proof}

\appendix

\section{Proof of Lemma~\ref{l.div}}

\begin{proof}[Proof of Lemma~\ref{l.div}:]

For $s\in [0,t]$, denote by $x_s=\phi^Y_s(x)$. Because 

\begin{eqnarray*}
\frac{d}{ds}\big|_{s_0}(\ln \det \Phi_s|_{T_xM}) & = & \lim_{\Delta s \to 0}\frac{1}{\Delta s} (\ln \det \Phi_{s_0+\Delta s}|_{T_xM}-\ln \det \Phi_{s_0}|_{T_xM}) \\
& = & \lim_{\Delta s \to 0}\frac{1}{\Delta s} (\ln \det \Phi_{\Delta s}|_{T_{x_{s_0}}M}) \\
& = & \lim_{\Delta s \to 0}\frac{1}{\Delta s} \ln \det 
\left| \begin{array}{cc}
\frac{\partial f_{\Delta s}}{\partial x}|_{x_{s_0}} & \frac{\partial g_{\Delta s}}{\partial x}|_{x_{s_0}} \\
\frac{\partial f_{\Delta s}}{\partial y}|_{x_{s_0}} &  \frac{\partial g_{\Delta s}}{\partial y}|_{x_{s_0}}\\
\end{array} \right| \\
& = & \frac{d}{d \Delta s}|_{\Delta s=0} \ln \big(\frac{\partial f_{\Delta s}}{\partial x}\frac{\partial g_{\Delta s}}{\partial y}-\frac{\partial g_{\Delta s}}{\partial x}\frac{\partial f_{\Delta s}}{\partial y}\big)\big|_{x_{s_0}}\\
& = &\frac{ \frac{d}{d \Delta s}|_{\Delta s=0} \big(\frac{\partial f_{\Delta s}(x_{s_0})}{\partial x}\frac{\partial g_{\Delta s}(x_{s_0})}{\partial y}-\frac{\partial g_{\Delta s}(x_{s_0})}{\partial x}\frac{\partial f_{\Delta s}(x_{s_0})}{\partial y}\big)}
{\det \Phi_{0}|_{T_{x_{s_0}}M}} \\
\end{eqnarray*}

Note that
$$\frac{d}{d \Delta s}|_{\Delta s =0}(\frac{\partial f_{\Delta s}(x_{s_0})}{\partial x})=
\frac{\partial}{\partial x} (\frac{d}{d \Delta x}|_{\Delta s =0} (f_{\Delta s}(x_{s_0})))=\frac{\partial}{\partial x} (Y_1(x_{s_0})).$$
Similarity, we have that 
$$\frac{d}{d \Delta s}|_{\Delta s =0}(\frac{\partial g_{\Delta s}(x_{s_0})}{\partial y})=\frac{\partial}{\partial x} (Y_2(x_{s_0})),$$
$$\frac{d}{d \Delta s}|_{\Delta s =0}(\frac{\partial g_{\Delta s}(x_{s_0})}{\partial x})=\frac{\partial}{\partial y} (Y_1(x_{s_0})),$$
$$\frac{d}{d \Delta s}|_{\Delta s =0}(\frac{\partial f_{\Delta s}(x_{s_0})}{\partial y})=\frac{\partial}{\partial y} (Y_2(x_{s_0})).$$

Because $\Phi_{0}|_{T_{x_{s_0}}M}=\ID$, we have
\begin{equation*}       
\left(                 
  \begin{array}{cc}  
  \frac{\partial f_{\Delta s}(x_{s_0})}{\partial x} & \frac{\partial g_{\Delta s}(x_{s_0})}{\partial x} \\
  \frac{\partial f_{\Delta s}(x_{s_0})}{\partial y} &  \frac{\partial g_{\Delta s}(x_{s_0})}{\partial y}\\
  \end{array}
\right)_{\Delta s =0}= \ID.
\end{equation*}

Combining the above argument, we show that 
$$\frac{d}{ds}\big|_{s_0}(\ln \det \Phi_s|_{T_xM}) = \frac{\partial{Y_1}}{\partial{x}}+\frac{\partial{Y_2}}{\partial{y}}=div_Y(x_{s_0})$$

Hence, 
$$\ln \det \Phi_t|_{T_xM}= \int_0^t \frac{d}{ds} (\ln \det \Phi_s|_{T_xM}) ds=\int_{0}^t div_Y(x_s) ds.$$
The proof is complete.

\end{proof}

\end{document}